\documentclass{article}

\usepackage{graphicx}
\usepackage{cite}
\usepackage{authblk}
\usepackage{amsthm}
\usepackage{amsmath}
\usepackage{amssymb}
\usepackage[left=3cm, right=3cm, lines=45, top=1.0in, bottom=1.0in]{geometry}
\usepackage[colorlinks=true]{hyperref}

\date{}
\title{A Brunn--Minkowski inequality for the Hessian eigenvalue in convex domain}
\author{ Jiahuan Li, Xi-Nan Ma, Paolo Salani}

\newcommand{\keywords}[1]{\par\quad\textbf{Keywords:} #1}

\begin{document}

\maketitle
\newtheorem{theorem}{Theorem}[section]
\newtheorem{definition}[theorem]{Definition}
\newtheorem{lemma}[theorem]{Lemma}
\newtheorem{corollary}[theorem]{Corollary}
\newtheorem{example}[theorem]{Example}
\newtheorem{proposition}[theorem]{Proposition}
\newtheorem{conjecture}[theorem]{Conjecture}
\newtheorem{remark}[theorem]{Remark}

\begin{abstract}
We use the deformation methods to obtain the strictly log concavity of solution of a class Hessian equation in bounded convex domain in $\mathbb{R}^{n}$, as an application we get the Brunn--Minkowski inequality for the Hessian eigenvalue and characterize the equality case in bounded strictly convex domain in $\mathbb{R}^{n}$.
\end{abstract}


\numberwithin{equation}{section}

\section{Introduction}

In 1976, in a legendary paper \cite{brascamp1976extensions}, Brascamp and Lieb established the log-concavity of the fundamental solution of diffusion equation with convex potential in bounded convex domain in $\mathbb{R}^{n}$. This in particular implies the log-concavity of the first Dirichlet eigenfunction of Laplace equation in convex domains, and also the Brunn--Minkowski inequality for the first eigenvalue, that is, 
\begin{equation}\label{eq:bm-laplace}
\lambda\bigl((1-t)K_{0}+tK_{1}\bigr)^{-1/2}\geq (1-t)\lambda(K_{0})^{-1/2}+t\lambda(K_{1})^{-1/2},
\end{equation}
where $t\in[0,1]$ and $K_{0}$, $K_{1}$ are nonempty convex bodies in $\mathbb{R}^{n}$. In fact, this inequality holds for all compact connected domain having sufficiently regular boundary. As for the classical Brunn-Minkowski inequality, the equality case of \eqref{eq:bm-laplace} has its own interest, and Jerison \cite{jerison1996direct} pointed out that it is related to uniqueness of the solution for the Minkowski problem about $\lambda$. In \cite{colesanti2005brunn}, Colesanti provides a new proof of \eqref{eq:bm-laplace} for convex bodies, and proves that equality holds if and only if $K_{0}$ is homothetic to $K_{1}$. Simultaneously, Salani \cite{salani2005brunn} proved the Brunn-Minkowski inequality for the Dirichlet eigenvalue of the Monge-Amp\`ere equation and showed that equality again holds if and only if the involved convex sets are homothetic.

Laplacian and Monge-Amp\`ere operators are the extremal cases (corresponding to $k=1$ and $k=n$) of a class of operators known as {\em Hessian operators}: for $k=1,\ldots,n$, and a $C^{2}$ function $u$, the $k$-th Hessian operator $\sigma_{k}(D^{2}u)$ is the $k$-th elementary symmetric function of the eigenvalues of the Hessian matrix $D^2u$ of $u$. More explicitly:
$$
\sigma_k(D^2u)=\sum_{1\leq i_1<\dots<i_k\leq n}\lambda_{i_1}\cdots\lambda_{i_k}\,,
$$
where $\lambda_1,\dots,\lambda_n$ are the eigenvalues of $D^2u$. Notice that $\sigma_k(D^2u)$ can be also defined as the sum of all principal $k\times k$ minors of $D^2u$.
The operator $\sigma_k$, for $k>1$, is fully nonlinear and it is elliptic only when restricted to a suitable class of admissible functions, called {\em $k$-convex functions}:
a function $u\in C^2(\Omega)$ is said {\em $k$-convex} in an open set $\Omega\subset\mathbb{R}^n$ if $\sigma_j(D^2u)\geq 0$ in $\Omega$ for $j=1,\dots,k$. 
Notice that $n$-convexity is equivalent to usual convexity in the class of $C^2$ functions, and it implies $k$-convexity for any $k=1,\dots,n$. Similarly, $k$-convexity, for $k=1,\dots, n-1$, can be defined also for $C^2$ sets in $\mathbb{R}^n$, using the symmetric functions of the principal curvatures. For sets, $(n-1)$-convexity coincides with usual convexity. Noticeably, the level sets of a $k$-convex functions are $(k-1)$-convex. 

Hessian operators give rise to Hessian equations. Dirichlet problems for Hessian equations have attracted great interest and have been deeply studied by many authors, following the seminal paper \cite{caffarelli1985dirichlet}.
In particular, given $k\in\{2,\dots,n-1\}$ and a smooth uniformly $(k-1)$-convex domain $\Omega\subset\mathbb{R}^n$, the eigenvalue problems for the Laplacian and Monge-Amp\'ere operators are generalized as follows:
\begin{equation}\label{eq:eigenvalue-problem}
\begin{cases}
\sigma_{k}(D^{2}u)=\lambda_k(\Omega)(-u)^{k},\quad u<0 & \text{in }\Omega,\\
u=0 & \text{on }\partial \Omega\,.
\end{cases}
\end{equation}
Wang \cite{wang1994class} proved that, for $1<k<n$, there exists a unique positive eigenvalue $\lambda(\Omega)$ such that this problems can be solved, and it has a unique, up to multiplication by a positive factor, negative $k$-convex solution $u\in C^{\infty}(\Omega)\cap C^{1,1}(\overline{\Omega})$. 
Equivalently we can define
\begin{equation}\label{eq:lambda-def}
\lambda_k(\Omega)=\inf\left\{
\frac{-\int_{\Omega}u\,\sigma_{k}(D^{2}u)\,dx}{\int_{\Omega}|u|^{k+1}\,dx\,:\, u\in\Phi_0^k(\Omega)}
\right\},
\end{equation}
where 
$$
\Phi_0^k=\left\{u\in C^{2}(\Omega)\cap C(\overline{\Omega})\,\text{ and }\,u\,\text{is $k$-convex in }\Omega\right\}
$$
Clearly this functional $\lambda(\Omega)$ is homogeneous of order $-2k$.

As we said, Laplace equation and Monge-Amp\'ere equation are just particular cases of $k$-Hessian equations, so it is natural to ask if results similar to the ones described at the beginning hold for all the $k$-Hessian equations. In particular, the questions are whether for $n\geq 3$ and $k=2,\dots,(n-1)$ the eigenvalue $\lambda_k$ satisfies a Brunn-Minkowski inequality, and whether the eigenfunction $u$ is log-convex (in the sense that $-\log(-u)$ is convex) or shares any other convexity property when $\Omega$ is convex.

A positive answer to these two questions has been given in \cite{ma2008convexity, liumaxu2010cc} and later in \cite{salani2012cc}, but only for $k=2$ and $n=3$. Here we treat the case $k=2$ for $n>3$, proving the following two theorems.

\begin{theorem}\label{thm:strict-convex}
Let $n\geq4$. Assume $K$ is a bounded uniformly convex smooth domain in $\mathbb{R}^{n}$, with everywhere strictly positive mean curvature of the boundary,
and let $u\in C^{\infty}(K)\cap C^{1,1}(\overline{K})$ is the unique (up to multiplication by a positive factor) admissible solution of
\begin{equation}\label{eq:s2-eigen}
\begin{cases}
\sigma_{2}(D^{2}u)=\lambda_2(K)(-u)^{2},\quad u<0 & \text{in }K,\\
u=0 & \text{on }\partial K\,.
\end{cases}
\end{equation}
Then $v=-\log(-u)$ is a convex function, with $D^2v>0$, in $K$.
\end{theorem}

\begin{theorem}\label{thm:bm}
Let $n\geq4$, and $t\in[0,1]$. Suppose $K_{0},K_{1}$ are bounded uniformly convex smooth domains in $\mathbb{R}^{n}$. Then the functional $\lambda$ satisfies the inequality
\begin{equation}\label{eq:bm-s2}
\lambda_2\bigl((1-t)K_{0}+tK_{1}\bigr)^{-1/4}\geq (1-t)\lambda_2(K_{0})^{-1/4}+t\lambda_2(K_{1})^{-1/4}.
\end{equation}
Moreover, equality holds if and only if $K_{0}$ and $K_1$ are homothetic.
\end{theorem}

%With this result we will prove the Brunn--Minkowski inequality for the positive eigenvalue of $S_{2}$ operator. Our method is from Colesanti \cite{colesanti2005brunn} and Salani \cite{salani2005brunn}.For the case of n=3, please refer to \cite{liumaxu2010cc}, \cite{salani2012cc} and \cite{ccq2025cc}.

Our method follows the lines of \cite{colesanti2005brunn, ma2008convexity, liumaxu2010cc, salani2012cc}, and it is based on a crucial algebraic property, see Proposition \ref{thm:main}, which establishes the concavity of a functional over the space of positive symmetric matrices. 
\medskip

Notice that the method is equally effective in treating more generic Dirichelt problems for Hessian equations and related functionals. In particular, as an example of applications, we  obtain the following theorem, which generalize to $n\geq 4$ the similar result from  \cite{ma2008convexity, liumaxu2010cc, salani2012cc}.
\begin{theorem}\label{thm:mainssss}
Suppose $u\in C^{\infty}(\Omega)$ is the admissible solution of \begin{equation}\label{eq:main-s2}
\sigma_{2}(\lambda(D^{2}u))=1\quad\text{in }\Omega,
\qquad
u=0\quad\text{on }\partial\Omega.
\end{equation}
where $\Omega$ is a bounded uniformly convex smooth domain in $\mathbb{R}^{n}$, $n\geq4$. Then
\[
v:=-( -u)^{1/2}
\]
is strictly convex, and the power $1/2$ is sharp.
\end{theorem}

Of course, we can also obtain a Brunn-Minkowski inequality for the corresponding {\em $k$-torsional rigidity}
$$
\tau_2(\Omega)=\left(\int_\Omega u\,dx\right)^2\,,
$$
where $u$ is the solution of \eqref{eq:main-s2}. That is,
{\em $\tau_2^{1/(2n+4)}$ is concave with respect to Minkowski addition of convex sets}. And from this, with a standard argument (see \cite{salani2012cc}), we can eventually get an Uhryson's type inequality for $\tau_2$, telling that its value is maximized by balls among convex sets with given mean width.

%The paper is organized as follows.
%In Section 2, we explain the strategy to prove Theorem \ref{thm:strict-convex} and Theorem \ref{thm:bm}. In Section 3, we prove a crucial lemma to conclude the proofs of the two theorems. In Section 4 we give further applications, like Theorem \ref{thm:mainssss}.

\section*{Acknowledgments}
Xinan Ma and Jiahuan Li were supported by the National Natural Science Foundation of China [grant number 2025YFA1017601]. 
Paolo Salani was partially supported by INdAM through GNAMPA and by the project "Geometric-Analytic Methods for PDEs and Applications (GAMPA)", funded by European Union - Next Generation EU  within the PRIN 2022 program (D.D. 104 - 02/02/2022 Ministero dell’Universit\`a e della Ricerca). 

\section{A crucial convexity property}

In this section we establish a crucial property for our method, Proposition \ref{thm:main}. Before, we recall some standard results about hyperbolic polynomials.

\subsection{G{\aa}rding's Quotient Concavity Theorem}

\begin{definition}
Let $p$ be a homogeneous real polynomial on a real vector space $V$. If there exists a direction $e\in V$ such that $p(e)>0$ and, for every $x\in V$, the one-variable polynomial
\[
t\mapsto p(x+te)
\]
has only real roots, then $p$ is called hyperbolic with respect to $e$. The connected component of $\{x:p(x)>0\}$ containing $e$ is called the hyperbolicity cone of $p$ and is denoted by $\Gamma_{p}$.
\end{definition}

The most important example here is $p(A)=\det A$, which is hyperbolic with respect to $I$, and whose hyperbolicity cone is exactly $S^{n}_{++}$.

\begin{theorem}[G{\aa}rding quotient concavity \cite{garding1959inequality,bauschke2001hyperbolic}]\label{thm:garding}
Let $p$ be a homogeneous hyperbolic polynomial of degree $m$, with hyperbolicity cone $\Gamma_{p}$. If $e\in\Gamma_{p}$, then
\[
x\mapsto \frac{p(x)}{D_{e}p(x)}
\]
is concave on $\Gamma_{p}$.
\end{theorem}

\begin{remark}
This theorem is the abstract version of a direct second derivative proof of concavity. Since $p$ has degree $m$ and $D_{e}p$ has degree $m-1$, the quotient is homogeneous of degree one. Thus one only needs to verify that its second variation is nonpositive. After normalizing the direction by adding a suitable multiple of the base point, the sign of the second derivative is controlled by a G{\aa}rding--Hodge type inequality. For $p=\det$, this is the mixed-discriminant form of the Alexandrov--Fenchel inequality.
\end{remark}

We also need the standard fact that directional derivatives preserve hyperbolicity.

\begin{theorem}[\cite{renegar2006hyperbolic})]\label{thm:directional-hyperbolic}
Let $p$ be hyperbolic with respect to $e$, with hyperbolicity cone $\Gamma_{p}$. If $v\in \Gamma_{p}$, then $D_{v}p$ is again hyperbolic, and its hyperbolicity cone contains $\Gamma_{p}$.
\end{theorem}

This follows from Rolle's theorem: along every line one obtains a real-rooted one-variable polynomial, and its derivative is again real-rooted. Directions in the closed cone are obtained by approximation from interior directions. See also \cite{renegar2006hyperbolic} for the role of derivative relaxations of hyperbolic polynomials.

\subsection{A crucial lemma}

We work on the cone of real symmetric positive definite matrices
\[
S^{n}_{++}=\{A=A^{T}>0\}.
\]
For a real symmetric matrix $M$, set
\[
\sigma_{2}(M)=\sum_{1\leq i<j\leq n}\lambda_{i}(M)\lambda_{j}(M),
\]
where $\lambda_{1}(M),\ldots,\lambda_{n}(M)$ are the eigenvalues of $M$. Equivalently, if
\[
\det(tI-M)=t^{n}-\sigma_{1}(M)t^{n-1}+\sigma_{2}(M)t^{n-2}-\cdots+(-1)^{n}\sigma_{n}(M),
\]

\begin{proposition}\label{thm:main}
Let $n\geq2$, $A,B\in S^{n}_{++}$, and $\alpha\in\mathbb{R}^{n}$ with $|\alpha|=1$. Set
\[
P=I_{n}-\alpha\alpha^{T}.
\]
Then
\begin{equation}\label{eq:main-ineq}
\frac{\operatorname{Tr}(P(A+B)^{-1})}{\sigma_{2}((A+B)^{-1})}
\geq
\frac{\operatorname{Tr}(PA^{-1})}{\sigma_{2}(A^{-1})}
+\frac{\operatorname{Tr}(PB^{-1})}{\sigma_{2}(B^{-1})}.
\end{equation}
\end{proposition}

Define
\[
\Phi(A)=\frac{\operatorname{Tr}(PA^{-1})}{\sigma_{2}(A^{-1})}.
\]
Since $(tA)^{-1}=t^{-1}A^{-1}$,
\[
\operatorname{Tr}(P(tA)^{-1})=t^{-1}\operatorname{Tr}(PA^{-1}),
\qquad
\sigma_{2}((tA)^{-1})=t^{-2}\sigma_{2}(A^{-1}).
\]
Thus
\[
\Phi(tA)=t\Phi(A),\qquad t>0.
\]
Therefore $\Phi$ is homogeneous of degree one. If $\Phi$ is concave on $S^{n}_{++}$, then
\[
\Phi(A+B)=2\Phi\left(\frac{A+B}{2}\right)
\geq 2\left(\frac{\Phi(A)+\Phi(B)}{2}\right)
=\Phi(A)+\Phi(B),
\]
which is precisely \eqref{eq:main-ineq}. Thus the main task is to prove that $\Phi$ is concave on $S^{n}_{++}$.

\subsubsection{Eliminating the Inverse Matrix}

Let $\sigma_{k}(A)$ denote the $k$-th elementary symmetric polynomial of the eigenvalues of $A$. In particular,
\[
\sigma_{n}(A)=\det A.
\]

\begin{lemma}\label{lem:inverse-remove}
For every $A\in S^{n}_{++}$,
\[
\frac{\operatorname{Tr}(PA^{-1})}{\sigma_{2}(A^{-1})}
=
\frac{\operatorname{Tr}(P\,\operatorname{adj} A)}{\sigma_{n-2}(A)}.
\]
\end{lemma}

\begin{proof}
Let the eigenvalues of $A$ be $\lambda_{1},\ldots,\lambda_{n}>0$. Then the eigenvalues of $A^{-1}$ are $\lambda_{1}^{-1},\ldots,\lambda_{n}^{-1}$, and hence
\[
\sigma_{2}(A^{-1})=\sum_{1\leq i<j\leq n}\frac{1}{\lambda_{i}\lambda_{j}}
=\frac{\sigma_{n-2}(A)}{\det A}.
\]
On the other hand,
\[
A^{-1}=\frac{\operatorname{adj} A}{\det A},
\]
so
\[
\operatorname{Tr}(PA^{-1})=\frac{\operatorname{Tr}(P\,\operatorname{adj} A)}{\det A}.
\]
Dividing the two identities gives the result.
\end{proof}

Thus it remains to study
\[
\Phi(A)=\frac{\operatorname{Tr}(P\,\operatorname{adj} A)}{\sigma_{n-2}(A)}.
\]

\subsubsection{Writing the Numerator and Denominator as Derivatives of One Polynomial}

Set
\[
b=\alpha\alpha^{T},\qquad P=I-b,
\]
and define
\[
R=b+\frac12P=\frac12(I+b).
\]
Notice that $R$ is positive definite, since $b$ is a rank-one orthogonal projection and the eigenvalues of $R$ are $1,1/2,\ldots,1/2$.

Let
\[
g(A)=D_{P}\det(A),
\]
where $D_{P}$ denotes the directional derivative in the direction $P$:
\[
D_{P}\det(A)=\left.\frac{d}{dt}\det(A+tP)\right|_{t=0}.
\]

\begin{lemma}\label{lem:numerator}
For every $A\in S^{n}_{++}$,
\[
g(A)=\operatorname{Tr}(P\,\operatorname{adj} A).
\]
\end{lemma}

\begin{proof}
The first derivative formula for the determinant gives
\[
D_{P}\det(A)=\det A\,\operatorname{Tr}(A^{-1}P).
\]
Since $A^{-1}=(\operatorname{adj}A)/\det A$, this equals $\operatorname{Tr}(P\,\operatorname{adj}A)$.
\end{proof}

Next we show that the denominator is also a directional derivative of the same polynomial.

\begin{lemma}\label{lem:denominator}
For every $A\in S^{n}_{++}$,
\[
D_{R}g(A)=\sigma_{n-2}(A).
\]
\end{lemma}

\begin{proof}
Since $g(A)=D_{P}\det(A)$,
\[
D_{R}g(A)=D_{R}D_{P}\det(A).
\]
Let $X=A^{-1}$. The second directional derivative formula for the determinant is
\[
D_{U}D_{V}\det(A)=\det A\left[
\operatorname{Tr}(XU)\operatorname{Tr}(XV)-\operatorname{Tr}(XUXV)
\right].
\]
Taking $U=P$ and $V=R$, we obtain
\[
D_{R}g(A)=\det A\left[
\operatorname{Tr}(XP)\operatorname{Tr}(XR)-\operatorname{Tr}(XPXR)
\right].
\]
Write
\[
s=\operatorname{Tr}X,\qquad c=\operatorname{Tr}(Xb)=\alpha^{T}X\alpha.
\]
Since $P=I-b$ and $R=\frac12(I+b)$,
\[
\operatorname{Tr}(XP)=s-c,\qquad
\operatorname{Tr}(XR)=\frac12(s+c).
\]
Therefore
\[
\operatorname{Tr}(XP)\operatorname{Tr}(XR)=\frac12(s-c)(s+c)
=\frac12(s^{2}-c^{2}).
\]
Moreover,
\[
\operatorname{Tr}(XPXR)=\frac12\operatorname{Tr}\bigl(X(I-b)X(I+b)\bigr).
\]
Expanding the matrix inside the trace gives
\[
X(I-b)X(I+b)=X^{2}+X^{2}b-XbX-XbXb.
\]
By cyclic invariance of the trace,
\[
\operatorname{Tr}(X^{2}b)=\operatorname{Tr}(XbX),
\]
so the middle two terms cancel after taking the trace. Since $b=\alpha\alpha^{T}$ is a rank-one projection,
\[
\operatorname{Tr}(XbXb)=(\alpha^{T}X\alpha)^{2}=c^{2}.
\]
Thus
\[
\operatorname{Tr}(XPXR)=\frac12\left(\operatorname{Tr}(X^{2})-c^{2}\right).
\]
It follows that
\[
\operatorname{Tr}(XP)\operatorname{Tr}(XR)-\operatorname{Tr}(XPXR)
=\frac12\left(s^{2}-\operatorname{Tr}(X^{2})\right)
=\sigma_{2}(X).
\]
Therefore
\[
D_{R}g(A)=\det A\,\sigma_{2}(A^{-1}).
\]
By Lemma \ref{lem:inverse-remove},
\[
\sigma_{2}(A^{-1})=\frac{\sigma_{n-2}(A)}{\det A},
\]
and hence $D_{R}g(A)=\sigma_{n-2}(A)$.
\end{proof}

Combining the previous lemmas,
\begin{equation}\label{eq:key-identification}
\Phi(A)=\frac{\operatorname{Tr}(P\,\operatorname{adj} A)}{\sigma_{n-2}(A)}
=\frac{g(A)}{D_{R}g(A)}.
\end{equation}
This identification is the key point of the proof.

\subsubsection{Proof of proposition\ref{thm:main}}

We have defined
\[
g(A)=D_{P}\det(A),\qquad P=I-\alpha\alpha^{T}.
\]
Since $P$ is semipositive definite, $P$ lies in the closure of $S^{n}_{++}$. Because $\det$ is hyperbolic on the positive definite cone. Approaching a semi positive definite matrix through a positive definite matrix, combined with the closure of real coefficient polynomials with respect to coefficients , Theorem \ref{thm:directional-hyperbolic} implies that
\[
g(A)=D_{P}\det(A)
\]
is also a hyperbolic polynomial. And that its hyperbolicity cone contains $S^{n}_{++}$ is obvious.

On the other hand,
\[
R=\frac12(I+\alpha\alpha^{T})>0,
\]
so $R\in S^{n}_{++}$. Since $S^{n}_{++}$ is contained in the hyperbolicity cone of $g$, we may apply Theorem \ref{thm:garding} to $g$ in the direction $R$. Hence
\[
A\mapsto \frac{g(A)}{D_{R}g(A)}
\]
is concave on $S^{n}_{++}$. By \eqref{eq:key-identification}, this function is exactly
\[
\Phi(A)=\frac{\operatorname{Tr}(PA^{-1})}{\sigma_{2}(A^{-1})}.
\]
Thus $\Phi$ is concave on $S^{n}_{++}$.

\subsection{The equal case of proposition\ref{thm:main}}
This subsection states the equality theory used below. 

\subsubsection{A strictly convex theorem}
We recall the parts of Bauschke--G{\"u}ler--Lewis--Sendov \cite{bauschke2001hyperbolic} that are used here. Let $X$ be a finite-dimensional real vector space, and let $p$ be a homogeneous polynomial of degree $m$ on $X$.
If $p$ is hyperbolic with respect to $d$, \cite{bauschke2001hyperbolic} define the characteristic-root map
\[
\lambda(x)=(\lambda_{1}(x),\ldots,\lambda_{m}(x))
\]
by
\[
p(x+td)=p(d)\prod_{i=1}^{m}(t+\lambda_{i}(x)).
\]
The corresponding hyperbolicity cone is denoted by $C(d)$.

\begin{definition}[\cite{bauschke2001hyperbolic} Definition 2.8: completeness]\label{def:complete}
A hyperbolic polynomial $p$ is called complete if
\[
\{x\in X:\lambda(x)=0\}=\{0\}.
\]
\end{definition}

\cite{bauschke2001hyperbolic} Fact 2.9 gives the equivalent descriptions
\[
\{x:\lambda(x)=0\}
=
\{x:x+C(d)=C(d)\}
=
\{x:p(tx+y)=p(y),\ \forall y\in X,\ \forall t\in\mathbb{R}\}.
\]
Therefore, for a hyperbolic polynomial $p$,
\[
p\ \text{is complete}
\quad\Longleftrightarrow\quad
\text{there is no nonzero }L\text{ such that }p(y+tL)=p(y)
\]
for all $y$ and $t$. Since
\[
D_{L}p\equiv0
\quad\Longleftrightarrow\quad
p(y+tL)=p(y),\qquad \forall y,\ t,
\]
we will use the following equivalent form:
\begin{equation}\label{eq:complete-equivalent}
p\ \text{is complete}
\quad\Longleftrightarrow\quad
\{L:D_{L}p\equiv0\}=\{0\}.
\end{equation}

\begin{theorem}[\cite{bauschke2001hyperbolic} Corollary 4.7, form used here]\label{thm:bgls}
Let $p$ be hyperbolic with respect to $d$, and assume $p(d)>0$. Let $C=C(d)$. For $c\in C$, define
\[
F(x)=-\log p(x),\qquad h_{c}(x)=-(\nabla F(x))(c).
\]
Then $F$ and $h_{c}$ are convex on $C$. Moreover, if $p$ is complete, then $F$ and $h_{c}$ are strictly convex on $C$.
\end{theorem}

Since
\[
F(x)=-\log p(x)
\]
implies
\[
-(\nabla F(x))(c)=\frac{D_{c}p(x)}{p(x)},
\]
the direct consequence used here is:
\begin{equation}\label{eq:bgls-direct}
p\ \text{complete},\ c\in C(d)
\quad\Longrightarrow\quad
x\mapsto\frac{D_{c}p(x)}{p(x)}
\ \text{is strictly convex on }C(d).
\end{equation}

\subsubsection{The Block Formula}
Through orthogonal transformation, we can assume $$P=\begin{bmatrix}I_{m}&0\\ 0&0\end{bmatrix}$$
where $m=n-1$.\\
Define the Newton transform of $C$ by
\[
T_{k}(C)=\sigma_{k}(C)I-\sigma_{k-1}(C)C+\sigma_{k-2}(C)C^{2}-\cdots+(-1)^{k}C^{k}.
\]
Then
\begin{equation}\label{eq:g-block}
g(A)=\operatorname{Tr}(P\,\operatorname{adj}A)
=c\,\sigma_{m-1}(C)-u^{T}T_{m-2}(C)u.
\end{equation}
Indeed,
\[
g(A)=D_{P}\det(A)
=\left.\frac{d}{dt}
\det
\begin{pmatrix}
C+tI_{m} & u\\
u^{T} & c
\end{pmatrix}\right|_{t=0}.
\]
When $C+tI_{m}$ is invertible, the Schur complement formula gives
\[
\det
\begin{pmatrix}
C+tI_{m} & u\\
u^{T} & c
\end{pmatrix}
=\det(C+tI_{m})
\left(c-u^{T}(C+tI_{m})^{-1}u\right).
\]
Differentiating at $t=0$, the first part contributes
\[
c\,D_{I}\det(C)=c\,\operatorname{Tr}(\operatorname{adj}C)=c\,\sigma_{m-1}(C),
\]
and the second part gives
\[
-D_{I}\left(\det(C)u^{T}C^{-1}u\right)
=-u^{T}T_{m-2}(C)u.
\]
Since both sides are polynomial identities, no invertibility assumption on $C$ is needed.

\subsubsection{The General Form of a Degeneracy Direction}

Let
\[
L=
\begin{pmatrix}
M & v\\
v^{T} & \ell
\end{pmatrix}.
\]
Assume $L\in L_{g}$, that is,
\[
D_{L}g(A)\equiv0.
\]
First set $u=0$. Then
\[
g(A)=c\,\sigma_{m-1}(C).
\]
Thus
\[
D_{L}g(A)=\ell\,\sigma_{m-1}(C)+c\,D_{M}\sigma_{m-1}(C).
\]
This is identically zero for all $C$ and $c$, so
\[
\ell=0,\qquad D_{M}\sigma_{m-1}(C)\equiv0.
\]
Since
\[
D_{M}\sigma_{m-1}(C)=\operatorname{Tr}(T_{m-2}(C)M),
\]
we obtain
\[
\operatorname{Tr}(T_{m-2}(C)M)=0,\qquad C\in S^{m}.
\]

\subsubsection{The Case $n\geq4$}\label{sec equal}

If $n\geq4$, then $m=n-1\geq3$. Expand near $C=I+sN$. Since
\[
D_{M}\sigma_{m-1}(I+sN)\equiv0,
\]
setting $s=0$ gives
\[
D_{M}\sigma_{m-1}(I)=0.
\]
But
\[
D_{M}\sigma_{m-1}(I)=(m-1)\operatorname{Tr}M,
\]
so
\[
\operatorname{Tr}M=0.
\]
Taking the coefficient of $s$, the second derivative formula gives
\[
D_{N}D_{M}\sigma_{m-1}(I)
=(m-2)\left(\operatorname{Tr}N\,\operatorname{Tr}M-\operatorname{Tr}(NM)\right).
\]
Since $\operatorname{Tr}M=0$,
\[
-(m-2)\operatorname{Tr}(NM)=0,\qquad N\in S^{m}.
\]
Because $m-2>0$ and $N$ is arbitrary, we get $M=0$.

We already know $M=0$ and $\ell=0$. Returning to the full formula \eqref{eq:g-block}, the direction $v$ corresponds to the variation of $u$, and hence
\[
D_{L}g(A)=-2v^{T}T_{m-2}(C)u.
\]
This vanishes for all $C,u$. Taking $C=I$, we have $T_{m-2}(I)=(m-1)I$, so
\[
-2(m-1)v^{T}u=0,\qquad u\in\mathbb{R}^{m}.
\]
Hence $v=0$. Therefore $L=0$, i.e.
\[
n\geq4\quad\Longrightarrow\quad L_{g}=\{0\}.
\]
Hence by theorem\ref{thm:bgls} we get following corollary.
\begin{corollary}\label{cor:Psi-strict}
When $n\geq4$,
\[
\Psi(A)=\frac{D_{R}g(A)}{g(A)}
\]
is strictly convex on $S^{n}_{++}$.
\end{corollary}
\subsubsection{Deriving the Higher-Dimensional Equality Condition}

\begin{theorem}[Higher-dimensional equality condition]\label{thm:equality}
Let $n\geq4$ and $A,B\in S^{n}_{++}$. Then
\[
\Phi(A+B)=\Phi(A)+\Phi(B)
\]
if and only if there exists $\lambda>0$ such that
\[
B=\lambda A.
\]
\end{theorem}

\begin{proof}
Set
\[
a=\Phi(A)>0,\qquad b=\Phi(B)>0.
\]
Define the normalized matrices
\[
\widehat{A}=\frac{A}{a},\qquad \widehat{B}=\frac{B}{b}.
\]
Since $\Phi$ is homogeneous of degree one,
\[
\Phi(\widehat{A})=1,\qquad \Phi(\widehat{B})=1.
\]
Equivalently,
\[
\Psi(\widehat{A})=1,\qquad \Psi(\widehat{B})=1.
\]
Let
\[
\theta=\frac{a}{a+b}\in(0,1).
\]
Then
\[
\frac{A+B}{a+b}
=\theta\widehat{A}+(1-\theta)\widehat{B}.
\]
By convexity of $\Psi$,
\[
\Psi\left(\frac{A+B}{a+b}\right)
=
\Psi(\theta\widehat{A}+(1-\theta)\widehat{B})
\leq
\theta\Psi(\widehat{A})+(1-\theta)\Psi(\widehat{B})
=1.
\]
Therefore
\[
\Phi\left(\frac{A+B}{a+b}\right)\geq1.
\]
Using homogeneity again,
\[
\Phi(A+B)\geq a+b=\Phi(A)+\Phi(B).
\]

If $\widehat{A}\neq\widehat{B}$, then the strict convexity of $\Psi$ gives
\[
\Psi(\theta\widehat{A}+(1-\theta)\widehat{B})
<
\theta\Psi(\widehat{A})+(1-\theta)\Psi(\widehat{B})
=1.
\]
Thus
\[
\Phi\left(\frac{A+B}{a+b}\right)>1,
\]
and hence
\[
\Phi(A+B)>\Phi(A)+\Phi(B).
\]
Consequently equality can occur only if
\[
\widehat{A}=\widehat{B},
\]
that is,
\[
\frac{A}{\Phi(A)}=\frac{B}{\Phi(B)}.
\]
Therefore
\[
B=\frac{\Phi(B)}{\Phi(A)}A.
\]
Let
\[
\lambda=\frac{\Phi(B)}{\Phi(A)}>0.
\]
Then $B=\lambda A$.

Conversely, if $B=\lambda A$ with $\lambda>0$, then
\[
\Phi(A+B)=\Phi((1+\lambda)A)
=(1+\lambda)\Phi(A)
=\Phi(A)+\Phi(\lambda A)
=\Phi(A)+\Phi(B).
\]
\end{proof}
\subsubsection{The Case $n=3$}
For the sake of completeness, we will explain why $n=3$ is a special case.
\\Here $m=n-1=2$, so
\[
\sigma_{m-1}(C)=\sigma_{1}(C)=\operatorname{Tr}C.
\]
Thus
\[
D_{M}\sigma_{1}(C)=\operatorname{Tr}M.
\]
The previous necessary conditions give only
\[
\ell=0,\qquad \operatorname{Tr}M=0.
\]
For the $v$ part, since $T_{m-2}=T_{0}=I$,
\[
D_{L}g(A)=-2v^{T}u.
\]
This vanishes for all $u$, so $v=0$. Consequently
\[
L_{g}=
\left\{
\begin{pmatrix}
M & 0\\
0 & 0
\end{pmatrix}
: M=M^{T},\ \operatorname{Tr}M=0
\right\}.
\]
In coordinate-free form,
\[
L_{g}=\{L=L^{T}:L\alpha=0,\ \operatorname{Tr}L=0\}.
\]
Hence $g$ is not complete.
\subsection{A convex theorem}

\begin{lemma}\label{lem:inverse-convex}
If $f(x)$ is a positive concave function in $\mathbb{R}^{n}$, then $f^{-1}$ is convex.
\end{lemma}

\begin{proof}
The condition of $f$ means
\[
f((1-t)x+ty)\geq(1-t)f(x)+tf(y),\qquad x,y\in\mathbb{R}^{n},\ t\in[0,1].
\]
Thus we have
\[
f((1-t)x+ty)^{-1}
\leq [(1-t)f(x)+tf(y)]^{-1}
\leq (1-t)f(x)^{-1}+tf(y)^{-1},
\]
where the last inequality comes from Jensen's inequality above. This says that $f^{-1}$ is convex.
\end{proof}

\begin{remark}\label{rem:inverse-equality}
If $f^{-1}$ is not strictly convex, then the two equalities in the proof of Lemma \ref{lem:inverse-convex} must hold at the same time. In particular $f((1-t)x+ty)=(1-t)f(x)+tf(y)$.
\end{remark}

\begin{proposition}\label{prop:matrix-convex}
Let $\alpha\in\mathbb{R}^{n}$ with $|\alpha|=1$, $P=I_{n}-\alpha\alpha^{T}$.Then the function
\[
f(A):=\frac{1}{\operatorname{Tr}(PA^{-1})}
\]
is concave in $A\in S^{n}_{++}$.
\end{proposition}
\begin{remark}
The fact that $1/\operatorname{Tr}(PA^{-1})$ is concave in $A\in S^{n}_{++}$ follows from the appendix in \cite{alvarez1997convexity}.
\end{remark}
Combining proposition\ref{thm:main} lemma\ref{lem:inverse-convex} and propsition\ref{prop:matrix-convex}, we can obtain following main theorems.
\begin{theorem}\label{thm1}
    Let $\alpha\in\mathbb{R}^{n}$ with $|\alpha|=1$, $P=I_{n}-\alpha\alpha^{T}$.Then for any $\lambda>0$ the function
\[
f(A):=\frac{ \sigma_{2}(A^{-1})}{\operatorname{Tr}(PA^{-1})}-\frac{\lambda}{\operatorname{Tr}(PA^{-1})}
\]
is convex in $A\in S^{n}_{++}$.
\end{theorem}
\section{Constant Rank Theorem}
\subsection{General Constant Rank Theorem}
For convenience, we state the version of the Bian--Guan structural constant rank theorem used here. It is a simplified form of Theorem 1.2 in \cite{bian2010structural}, adapted to the present setting, which refines the original microscopic convexity principle in \cite{bian2009microscopic}:

\begin{theorem}[Bian--Guan level-set form, application version]\label{thm:bg-level}
Let $\Omega\subset\mathbb{R}^{n}$ be a connected domain, and let $u\in C^{3,1}(\Omega)$ be a convex solution of
\[
F(D^{2}u,Du,u,x)=0,
\]
where
\[
F=F(r,p,z,x)\in C^{2,1}(S^{n}\times\mathbb{R}^{n}\times\mathbb{R}\times\Omega).
\]
Assume the following:
\begin{enumerate}
    \item $F$ is elliptic along the solution, namely
    \[
    \bigl(F^{ij}(D^{2}u,Du,u,x)\bigr)>0,
    \qquad
    F^{ij}:=\frac{\partial F}{\partial r_{ij}}.
    \]
    \item Along the solution, the nondegeneracy condition holds:
    \[
    F(0,Du(x),u(x),x)\neq0.
    \]
    \item For each fixed relevant gradient $p$, the sublevel set
    \[
    \Gamma_{F}(p)
    :=
    \left\{
    (A,z,x)\in S^{n}_{++}\times\mathbb{R}\times\Omega:
    F(A^{-1},p,z,x)\leq0
    \right\}
    \]
    is locally convex near the relevant points.
\end{enumerate}
Then
\[
\operatorname{rank}D^{2}u(x)
\]
is constant in $\Omega$.
\end{theorem}
\subsection{A Constant Rank Theorem}

In this subsection we establish a constant rank theorem for the convex solution of the related nonlinear elliptic equation.

In what follows, $S^{n}$ denotes the set of the symmetric $n\times n$ matrices, and $S^{n}_{+}$ ($S^{n}_{++}$) is the subset of the semipositive (positive) definite matrices.

Let $K\subset\mathbb{R}^{n}$ be any bounded domain. Note that if let $v=-\log(-u)$, then Eq. \eqref{eq:s2-eigen} is equivalent to
\begin{equation}\label{eq:v-equation}
\begin{cases}
\sigma_{2}(D^{2}v)-\operatorname{Tr}(P(Dv)D^{2}v)=\lambda(K) & \text{in }K,\\
v(x)\to+\infty & x\to\partial K,
\end{cases}
\end{equation}
where $\operatorname{Tr}(A)$ denotes the trace of matrix $A$, and $P(\nabla v)=(P_{ij})$ is a matrix with
\begin{equation}\label{eq:Pij}
P_{ij}=|\nabla v|^{2}\delta_{ij}-v_{i}v_{j},\quad i,j=1,2,...,n.
\end{equation}
It follows that $P$ is semipositive definite, and $\operatorname{Tr}(PD^{2}v)=\sum_{i,j=1}^{n}P_{ij}v_{ij}$. By a simple observation, if $u$ is an admissible solution for Eq. \eqref{eq:s2-eigen}, then $v$ is an admissible solution for Eq. \eqref{eq:v-equation}.

One of main ingredients of the proof of Theorem \ref{thm:strict-convex} is a constant rank theorem, which states as follows (see e.g. \cite{caffarelli1985convexity} and \cite{korevaar1987convex}). Some related results have been obtained in \cite{ma2008convexity}.

\begin{theorem}\label{thm:revised}
Let $K\subset\mathbb{R}^{n}$ be a connected domain. Let
\[
v\in C^{4}(K),\qquad D^{2}v(x)\geq0,
\]
and assume that $v$ satisfies
\begin{equation}\label{eq:theorem-equation}
\sigma_{2}(D^{2}v)-\operatorname{Tr}(P(Dv)D^{2}v)=\lambda(K)\quad\text{in }K,
\end{equation}
where
\[
P(Dv)=|Dv|^{2}I-Dv\otimes Dv.
\]
Assume further that $v$ is admissible, namely
\begin{equation}\label{eq:admissible}
W(x):=D^{2}v(x)-Dv(x)\otimes Dv(x)\in\Gamma_{2}
\quad\text{for all }x\in K.
\end{equation}
Then
\[
\operatorname{rank}D^{2}v(x)
\]
is constant in $K$.
\end{theorem}
\subsection{Proof of Theorem \ref{thm:revised}}

\subsubsection{Writing the Equation in a Form Suitable for the Constant Rank Theorem}

Define
\begin{equation}\label{eq:F-def}
F(r,p):=\sigma_{2}(r)-\operatorname{Tr}(P(p)r)-\lambda(K),
\qquad r\in S^{n},\quad p\in\mathbb{R}^{n},
\end{equation}
where
\[
P(p)=|p|^{2}I-p\otimes p.
\]
Then \eqref{eq:theorem-equation} is equivalent to
\[
F(D^{2}v,Dv)=0.
\]
The function $F$ is independent of $z$ and $x$, so these variables will not produce extra terms when Theorem \ref{thm:bg-level} is applied.

We also need the elementary identity
\begin{equation}\label{eq:s2-identity}
\sigma_{2}(r-p\otimes p)=\sigma_{2}(r)-\operatorname{Tr}(P(p)r).
\end{equation}
Indeed, from
\[
\sigma_{2}(M)=\frac12\left((\operatorname{Tr}M)^{2}-\operatorname{Tr}(M^{2})\right),
\]
and since $p\otimes p$ has rank one and $S_{2}(p\otimes p)=0$, we get
\[
\sigma_{2}(r-p\otimes p)
=\sigma_{2}(r)-|p|^{2}\operatorname{Tr}r+p^{T}rp.
\]
On the other hand,
\[
\operatorname{Tr}(P(p)r)
=\operatorname{Tr}\bigl((|p|^{2}I-p\otimes p)r\bigr)
=|p|^{2}\operatorname{Tr}r-p^{T}rp.
\]
Therefore \eqref{eq:s2-identity} holds. Consequently,
\begin{equation}\label{eq:F-alternative}
F(r,p)=\sigma_{2}(r-p\otimes p)-\lambda(K).
\end{equation}

\subsubsection{Verification of Ellipticity}

Differentiate with respect to $r$. Let
\[
T_{1}(N):=\sigma_{1}(N)I-N=(\operatorname{Tr}N)I-N
\]
be the first Newton tensor. Since
\[
\frac{\partial \sigma_{2}}{\partial r_{ij}}(r)=T_{1}(r)_{ij},
\]
\eqref{eq:F-def} gives
\begin{equation}\label{eq:Fij}
F^{ij}(r,p)=T_{1}(r)_{ij}-P(p)_{ij}.
\end{equation}
Using
\[
T_{1}(r-p\otimes p)=T_{1}(r)-P(p),
\]
we obtain, along the solution,
\begin{equation}\label{eq:Fij-solution}
F^{ij}(D^{2}v,Dv)
=T_{1}(D^{2}v-Dv\otimes Dv)_{ij}
=T_{1}(W)_{ij}.
\end{equation}
By the admissibility assumption \eqref{eq:admissible}, $W\in\Gamma_{2}$. If the eigenvalues of $W$ are $\mu=(\mu_{1},\ldots,\mu_{n})\in\Gamma_{2}$, then the eigenvalues of $T_{1}(W)$ are
\[
\frac{\partial\sigma_{2}}{\partial\mu_{i}}(\mu)
=\sum_{j\neq i}\mu_{j},
\qquad i=1,\ldots,n.
\]
These quantities are strictly positive in $\Gamma_{2}$. Hence
\begin{equation}\label{eq:elliptic}
\bigl(F^{ij}(D^{2}v,Dv)\bigr)=T_{1}(W)>0.
\end{equation}
This is precisely the ellipticity condition required by the Bian--Guan constant rank theorem.

\subsubsection{The Nondegeneracy Condition}

For every $p\in\mathbb{R}^{n}$,
\[
F(0,p)=\sigma_{2}(0)-\operatorname{Tr}(P(p)0)-\lambda(K)=-\lambda(K).
\]
Since the Hessian eigenvalue satisfies $\lambda(K)>0$,
\begin{equation}\label{eq:nondegenerate}
F(0,p)=-\lambda(K)\neq0.
\end{equation}
Thus the nondegeneracy condition in Theorem \ref{thm:bg-level} is satisfied.

\subsubsection{Convexity of the Sublevel Set: The Case $p\neq0$}

Fix $p\neq0$. Set
\[
\alpha=\frac{p}{|p|},
\qquad
Q_{\alpha}=I-\alpha\otimes\alpha.
\]
Then
\[
P(p)=|p|^{2}Q_{\alpha}.
\]
Consider the sublevel set
\begin{equation}\label{eq:Gamma-p}
\Gamma_{p}:=\{A\in S^{n}_{++}:F(A^{-1},p)\leq0\}.
\end{equation}
By \eqref{eq:F-def}, the condition $F(A^{-1},p)\leq0$ is equivalent to
\[
\sigma_{2}(A^{-1})-|p|^{2}\operatorname{Tr}(Q_{\alpha}A^{-1})-\lambda(K)\leq0.
\]
Since $A^{-1}>0$, $Q_{\alpha}\geq0$, and $Q_{\alpha}\neq0$,
\begin{equation}\label{eq:positive-denominator}
\operatorname{Tr}(Q_{\alpha}A^{-1})>0.
\end{equation}
Thus the previous inequality is equivalent to
\begin{equation}\label{eq:H-sublevel}
H_{\alpha,\lambda}(A)
:=
\frac{\sigma_{2}(A^{-1})-\lambda(K)}
{\operatorname{Tr}(Q_{\alpha}A^{-1})}
\leq |p|^{2}.
\end{equation}

using theorem\ref{thm1}, the function
\begin{equation}\label{eq:H-convex}
A\mapsto
\frac{\sigma_{2}(A^{-1})}{\operatorname{Tr}(Q_{\alpha}A^{-1})}
-\frac{\lambda}{\operatorname{Tr}(Q_{\alpha}A^{-1})}
=
\frac{\sigma_{2}(A^{-1})-\lambda}{\operatorname{Tr}(Q_{\alpha}A^{-1})}
\end{equation}
is convex on $S^{n}_{++}$. Therefore $\Gamma_{p}$ is a sublevel set of the convex function $H_{\alpha,\lambda}$, and is consequently convex, in particular locally convex.

\subsubsection{Convexity of the Sublevel Set: The Case $p=0$}

When $p=0$, the vector $\alpha=p/|p|$ is not defined. This is exactly the singular point in the division-based proof. We handle it by a limiting intersection argument.

Fix an arbitrary unit vector $\alpha\in\mathbb{R}^{n}$ and set
\[
p_{\varepsilon}=\varepsilon\alpha,\qquad \varepsilon>0.
\]
By the previous step, $\Gamma_{p_{\varepsilon}}$ is convex for every $\varepsilon>0$. We claim that
\begin{equation}\label{eq:intersection}
\Gamma_{0}=\bigcap_{\varepsilon>0}\Gamma_{p_{\varepsilon}}.
\end{equation}
Indeed, $A\in\Gamma_{p_{\varepsilon}}$ is equivalent to
\[
\sigma_{2}(A^{-1})-\lambda(K)
\leq
\varepsilon^{2}\operatorname{Tr}(Q_{\alpha}A^{-1}).
\]
If $A\in\bigcap_{\varepsilon>0}\Gamma_{p_{\varepsilon}}$, then this inequality holds for every $\varepsilon>0$. Letting $\varepsilon\to0$ yields
\[
\sigma_{2}(A^{-1})-\lambda(K)\leq0,
\]
which is exactly $A\in\Gamma_{0}$.

Conversely, if $A\in\Gamma_{0}$, then
\[
\sigma_{2}(A^{-1})-\lambda(K)\leq0.
\]
Since
\[
\varepsilon^{2}\operatorname{Tr}(Q_{\alpha}A^{-1})\geq0,
\]
we immediately get
\[
\sigma_{2}(A^{-1})-\lambda(K)
\leq
\varepsilon^{2}\operatorname{Tr}(Q_{\alpha}A^{-1}),
\]
so $A\in\Gamma_{p_{\varepsilon}}$ for every $\varepsilon>0$. This proves \eqref{eq:intersection}.

The intersection of any family of convex sets is convex. Hence $\Gamma_{0}$ is convex. Therefore, for every $p\in\mathbb{R}^{n}$, the set
\[
\Gamma_{p}
=
\{A\in S^{n}_{++}:F(A^{-1},p)\leq0\}
\]
is convex.

\subsubsection{Application of the Bian--Guan Level-Set Constant Rank Theorem}

Since $F$ is independent of $z$ and $x$, convexity in the $A$ variable implies the local convexity of
\[
\{(A,z,x)\in S^{n}_{++}\times\mathbb{R}\times K:
F(A^{-1},p,z,x)\leq0\}
\]
in the variables $(A,z,x)$. Combining this with the ellipticity verified in \eqref{eq:elliptic}, the nondegeneracy condition \eqref{eq:nondegenerate}, the convexity $D^{2}v\geq0$, and the regularity $v\in C^{4}(K)$, all assumptions of Theorem \ref{thm:bg-level} are satisfied. Hence
\[
\operatorname{rank}D^{2}v(x)
\]
is constant in the connected domain $K$. This proves Theorem \ref{thm:revised}.
\subsection{proof of theorem\ref{thm:mainssss}}
Through the same discussion as \cite{ma2008convexity}, If we can prove the following constant rank theorem, then Theorem\ref{thm:mainssss} is proven.
\begin{theorem}\label{lem:constant-rank}
Let $u\in C^{4}(\Omega)$ be an admissible solution of 
\begin{equation}\label{eq:main-s2}
\sigma_{2}(\lambda(D^{2}u))=1\quad\text{in }\Omega,
\qquad
u=0\quad\text{on }\partial\Omega.
\end{equation}
where $\Omega\subset\mathbb{R}^{n},n\geq4$ is any domain. If
\[
v:=-( -u)^{1/2}
\]
is a convex function, i.e. the Hessian matrix $W=\{v_{ij}\}$ is semipositive in $\Omega$, then $(v_{ij})$ has constant rank in $\Omega$.
\end{theorem}
\subsubsection{The Transformed Equation}

Since
\[
v=-( -u)^{1/2},
\]
we have
\[
u=-v^{2}.
\]
Write
\[
r=D^{2}v,\qquad p=Dv,\qquad z=v<0.
\]
Then
\[
u_{i}=-2vv_{i},
\]
and
\[
u_{ij}=-2(vv_{ij}+v_{i}v_{j}).
\]
In matrix form,
\[
D^{2}u=-2(zr+p\otimes p).
\]
Since $S_{2}$ is homogeneous of degree two, the equation $S_{2}(D^{2}u)=1$ becomes
\[
4\sigma_{2}(zr+p\otimes p)=1.
\]
For a matrix $r$ and a vector $p$,
\begin{equation}\label{eq:s2-identity}
\sigma_{2}(zr+p\otimes p)=z^{2}\sigma_{2}(r)+z\,\operatorname{Tr}(P(p)r),
\end{equation}
where
\[
P(p)=|p|^{2}I-p\otimes p.
\]
Indeed, using
\[
\sigma_{2}(M)=\frac12\left((\operatorname{Tr}M)^{2}-\operatorname{Tr}(M^{2})\right)
\]
and $\sigma_{2}(p\otimes p)=0$, one obtains
\[
\sigma_{2}(zr+p\otimes p)
=z^{2}\sigma_{2}(r)+z\left(|p|^{2}\operatorname{Tr}r-p^{T}rp\right),
\]
and
\[
|p|^{2}\operatorname{Tr}r-p^{T}rp
=\operatorname{Tr}((|p|^{2}I-p\otimes p)r)
=\operatorname{Tr}(P(p)r).
\]
Thus the equation can be written as
\begin{equation}\label{eq:F}
F(r,p,z)=0,
\end{equation}
where
\begin{equation}\label{eq:F-def}
F(r,p,z):=z^{2}\sigma_{2}(r)+z\,\operatorname{Tr}(P(p)r)-\frac14.
\end{equation}
The proof of ellipticity and nondegeneracy of $F$ is similar, we only prove sublevel-set convexity.
\subsubsection{Sublevel-Set Convexity}

Fix $p\in\mathbb{R}^{n}$. We prove that
\begin{equation}\label{eq:level-set}
\{(A,z)\in S^{n}_{++}\times(-\infty,0):F(A^{-1},p,z)\leq0\}
\end{equation}
is convex.

Set
\[
s=-z>0.
\]
Substituting $r=A^{-1}$ into \eqref{eq:F-def}, the inequality $F(A^{-1},p,z)\leq0$ becomes
\begin{equation}\label{eq:basic-level}
s^{2}\sigma_{2}(A^{-1})-s\,\operatorname{Tr}(P(p)A^{-1})\leq\frac14.
\end{equation}

\subsubsection{The case $p\neq0$}

Assume first that $p\neq0$. Let
\[
\alpha=\frac{p}{|p|},\qquad Q_{\alpha}=I-\alpha\otimes\alpha.
\]
Then
\[
P(p)=|p|^{2}Q_{\alpha}.
\]
Set
\[
B=\frac{A}{s},\qquad A=sB.
\]
Then
\[
A^{-1}=\frac1s B^{-1},
\]
and hence
\[
s^{2}\sigma_{2}(A^{-1})=\sigma_{2}(B^{-1}),
\qquad
s\,\operatorname{Tr}(Q_{\alpha}A^{-1})
=\operatorname{Tr}(Q_{\alpha}B^{-1}).
\]
Thus \eqref{eq:basic-level} is equivalent to
\begin{equation}\label{eq:B-level}
\sigma_{2}(B^{-1})-|p|^{2}\operatorname{Tr}(Q_{\alpha}B^{-1})\leq\frac14.
\end{equation}
Since $B^{-1}>0$, $Q_{\alpha}\geq0$, and $Q_{\alpha}\neq0$,
\[
\operatorname{Tr}(Q_{\alpha}B^{-1})>0.
\]
Therefore \eqref{eq:B-level} is equivalent to
\[
\frac{\sigma_{2}(B^{-1})-\frac14}{\operatorname{Tr}(Q_{\alpha}B^{-1})}
\leq |p|^{2}.
\]
We use the theorem\ref{thm1}, for every unit vector $\alpha$ and every $\lambda>0$, the function
\begin{equation}\label{eq:matrix-convexity}
B\mapsto
\frac{\sigma_{2}(B^{-1})-\lambda}{\operatorname{Tr}(Q_{\alpha}B^{-1})}
\end{equation}
is convex on $S^{n}_{++}$. Taking $\lambda=1/4$, we obtain that
\[
K_{p}:=
\left\{
B\in S^{n}_{++}:
\frac{\sigma_{2}(B^{-1})-\frac14}{\operatorname{Tr}(Q_{\alpha}B^{-1})}
\leq |p|^{2}
\right\}
\]
is convex.

The level set \eqref{eq:level-set} can now be written as
\[
C_{p}=\{(A,z):s=-z>0,\ A=sB,\ B\in K_{p}\}.
\]
If $(A_{i},z_{i})\in C_{p}$, write
\[
s_{i}=-z_{i}>0,\qquad B_{i}=A_{i}/s_{i}\in K_{p},
\qquad i=1,2.
\]
For $0\leq\theta\leq1$, set
\[
A_{\theta}=\theta A_{1}+(1-\theta)A_{2},\qquad
z_{\theta}=\theta z_{1}+(1-\theta)z_{2},
\]
and
\[
s_{\theta}=-z_{\theta}=\theta s_{1}+(1-\theta)s_{2}>0.
\]
Then
\[
\frac{A_{\theta}}{s_{\theta}}
=\frac{\theta s_{1}}{s_{\theta}}B_{1}
+\frac{(1-\theta)s_{2}}{s_{\theta}}B_{2}.
\]
The coefficients are nonnegative and sum to one. Since $K_{p}$ is convex,
\[
A_{\theta}/s_{\theta}\in K_{p}.
\]
Thus $(A_{\theta},z_{\theta})\in C_{p}$. Hence the level set is convex when $p\neq0$.

\subsubsection{The case $p=0$}

When $p=0$, \eqref{eq:basic-level} becomes
\[
s^{2}\sigma_{2}(A^{-1})\leq\frac14.
\]
With $B=A/s$, this is
\[
\sigma_{2}(B^{-1})\leq\frac14.
\]
Define
\[
K_{0}=\left\{B\in S^{n}_{++}:\sigma_{2}(B^{-1})\leq\frac14\right\}.
\]
Fix any unit vector $\alpha$ and set $p_{\varepsilon}=\varepsilon\alpha$. Then
\[
K_{p_{\varepsilon}}
=
\left\{
B:
\sigma_{2}(B^{-1})-\varepsilon^{2}\operatorname{Tr}(Q_{\alpha}B^{-1})
\leq\frac14
\right\}.
\]
We claim that
\[
K_{0}=\bigcap_{\varepsilon>0}K_{p_{\varepsilon}}.
\]
If $B\in K_{0}$, then
\[
\sigma_{2}(B^{-1})-\varepsilon^{2}\operatorname{Tr}(Q_{\alpha}B^{-1})
\leq \sigma_{2}(B^{-1})\leq\frac14,
\]
so $B\in K_{p_{\varepsilon}}$ for every $\varepsilon>0$. Conversely, if $B\in K_{p_{\varepsilon}}$ for every $\varepsilon>0$, then
\[
\sigma_{2}(B^{-1})\leq\frac14+\varepsilon^{2}\operatorname{Tr}(Q_{\alpha}B^{-1})
\]
for every $\varepsilon>0$. Letting $\varepsilon\downarrow0$, we obtain $B\in K_{0}$.

Since every $K_{p_{\varepsilon}}$ is convex and arbitrary intersections of convex sets are convex, $K_{0}$ is convex. The same scaling argument used above then proves the convexity of the level set \eqref{eq:level-set} when $p=0$.

\subsubsection{Conclusion}

We have verified the three hypotheses needed for the Bian--Guan level-set constant rank theorem \cite{bian2010structural}, which refines the original microscopic convexity principle in \cite{bian2009microscopic}:
\begin{enumerate}
    \item ellipticity, namely $(F^{ij})>0$;
    \item nondegeneracy, namely $F(0,p,z)=-1/4\neq0$;
    \item for each fixed $p$, convexity of the set
    \[
    \{(A,z):A\in S^{n}_{++},\ z<0,\ F(A^{-1},p,z)\leq0\}.
    \]
\end{enumerate}
Since $D^{2}v\geq0$ and $\Omega$ is connected, the Bian--Guan theorem applies and gives
\[
\operatorname{rank}D^{2}v(x)
\]
constant in $\Omega$.

\section{Proof of Theorem \ref{thm:strict-convex}}

In this section, we shall use the deformation process to prove Theorem \ref{thm:strict-convex}. The constant rank theorem is a very powerful tool to produce strict convex solution for nonlinear elliptic equation, see for example Caffarelli and Friedman \cite{caffarelli1985convexity}, Korevaar and Lewis \cite{korevaar1987convex} and Guan and Ma \cite{guan2003christoffel}. Here we follow the approach of Korevaar and Lewis \cite{korevaar1987convex} and Ma and Xu \cite{ma2008convexity} to get the result. First we need understand the radial solution of Eq. \eqref{eq:s2-eigen} defined on ball, along the idea of McCuan \cite{mccuan2002concavity}. Then we study the geometrical properties of the solution near the convex boundary.

\begin{lemma}\label{lem:ball}
Let $B_{R}(o)$ be unit ball in $\mathbb{R}^{n}$ with radius $R>0$. Let $u\in C^{\infty}(B_{R})\cap C^{1,1}(\overline{B}_{R})$ be the eigenfunction for Eq. \eqref{eq:s2-eigen} in $B_{R}(o)$. Then $v=-\log(-u)$ is a strict convex function in $B_{R}(o)$.
\end{lemma}

\begin{proof}
By the uniqueness of solution for Eq. \eqref{eq:s2-eigen} up to a constant factor, we know the solution $u$ is a radial function. We set
\[
u(x)=\varphi(|x|)=\varphi(r),\qquad r=|x|,
\]
where $r\in[0,R]$, $\varphi(r)<0$ for $r\in[0,R)$. Then $\varphi$ is an increasing function in $(0,R)$ and $\varphi'(0)=\varphi(R)=0$.

Since
\[
\frac{\partial r}{\partial x_{i}}=\frac{x_{i}}{r},\qquad
\frac{\partial^{2}r}{\partial x_{i}\partial x_{j}}=-r^{-3}x_{i}x_{j}+r^{-1}\delta_{ij},
\]
it follows that
\[
u_{ij}=(\varphi''r^{-2}-\varphi'r^{-3})x_{i}x_{j}+\varphi'r^{-1}\delta_{ij},
\]
and
\begin{equation}\label{eq:s2-radial}
\sigma_{2}(D^{2}u)=(n-1)\varphi'\varphi''r^{-1}+\frac{\left( n-1 \right) \left( n-2 \right)}{2}(\varphi')^{2}r^{-2}.
\end{equation}
For $v=-\log(-\varphi)$, we have
\[
\varphi'=e^{-v}v',\qquad \varphi''=e^{-v}\bigl(v''-(v')^{2}\bigr).
\]
So Eq. \eqref{eq:s2-eigen} on $u$ transforms to the following equation on $v$:
\begin{equation}\label{eq:radial-v}
(n-1)rv'v''-(n-1)r(v')^{3}+\frac{\left( n-1 \right) \left( n-2 \right)}{2}(v')^{2}=\lambda r^{2},\qquad 0<r<R.
\end{equation}
It follows that
\[
v'(0)=0,\qquad v'(r)>0\quad\text{for }0<r<R,
\]
and $v''(0)\geq0$. Let $L=\lim_{r\to0^{+}}v'(r)/r$, then by L'Hopital rule,
\[
L=\lim_{r\to0^{+}}v''(r)=v''(0).
\]
From \eqref{eq:radial-v}, it follows that
\begin{equation}\label{eq:vpp0}
v''(0)=\sqrt{\frac{2\lambda}{n(n-1)}}>0.
\end{equation}
Assume by contradiction the existence of a smallest positive $r_{0}$ for which $v''(r_{0})=0$. We know $v'(r)>0$ for $0<r\leq R$ and $v''(r)>0$ for $0<r<r_{0}$. Differentiating \eqref{eq:radial-v} and evaluating at $r=r_{0}$, we obtain
\[
v'''(r_{0})=\frac{2\lambda}{(n-1)v'(r_{0})}+\frac{(v'(r_{0}))^{2}}{r_{0}}>0,
\]
which contradicts the sign of $v''$. Hence $v$ is strictly convex in $[0,R)$ and the lemma is proven.
\end{proof}

Now we state the following well-known boundary convexity lemma, see for example Caffarelli and Friedman \cite{caffarelli1985convexity} or Korevaar \cite{korevaar1983convex}.

\begin{lemma}\label{lem:boundary}
Let $\Omega\subset\mathbb{R}^{n}$ be smooth, bounded and strictly convex. Let $u\in C^{\infty}(\Omega)\cap C^{1,1}(\overline{\Omega})$ satisfies
\begin{equation}\label{eq:boundary-data}
u<0\ \text{in }\Omega,\qquad u=0\ \text{and}\ Du\cdot\nu>0\ \text{on }\partial\Omega,
\end{equation}
where $\nu$ is the exterior normal to $\partial\Omega$. Let
\[
\Omega_{\varepsilon}=\{x\in\Omega:d(x,\partial\Omega)>\varepsilon\}
\]
and let $v=f(u)$. Then for small enough $\varepsilon>0$ the function $v$ is strictly convex in a boundary strip $\Omega\setminus\Omega_{\varepsilon}$ if $f$ satisfies
\[
f'>0,\qquad f''>0,\qquad \lim_{u\to0^{-}}\frac{f'}{f''}=0.
\]
\end{lemma}

\begin{remark}
In Korevaar \cite{korevaar1983convex}, the function $u\in C^{2}(\overline{\Omega})$. But we can follow the calculation in Caffarelli and Friedman \cite{caffarelli1985convexity} to know the similar result is true in our case $u\in C^{\infty}(\Omega)\cap C^{1,1}(\overline{\Omega})$.
\end{remark}

\begin{proof}[Proof of Theorem \ref{thm:strict-convex}]
If $K$ is the ball $B_{R}(o)$, by Lemma \ref{lem:ball}, for the solution $u$ of \eqref{eq:s2-eigen}, we have $v=-\log(-u)$ is a strict convex function in $B_{R}(o)$. For an arbitrary bounded strict convex domain $K$, set
\[
K_{t}=(1-t)B_{R}(o)+tK,\qquad 0\leq t\leq1.
\]
Then from the theory of convex bodies (see for example \cite{schneider1993convex} and \cite{zhu2002mean}), we can deform $B_{R}(o)$ continuously into $K$ by the family $(K_{t})$, $0\leq t<1$, of strictly convex domain in such a way that $\partial K_{t}\to\partial K_{s}$ as $t\to s$ in the sense of Hausdorff distance, whenever $0\leq s\leq1$. And the deformation also is chosen so that $\partial K_{t}$, $0\leq t<1$, can be locally represented for some $\alpha$, $0<\alpha<1$, by a function whose norm in the space $C^{2,\alpha}$ of functions with Holder continuous second derivatives depends only on $\delta$, whenever $0<t\leq\delta<1$.

Suppose $u_{t}\in C^{\infty}(K_{t})\cap C^{1,1}(\overline{K}_{t})$ is the admissible solution of \eqref{eq:s2-eigen}, $v_{t}:=-\log(-u_{t})$ and $H_{t}$ is the corresponding Hessian matrix of $v_{t}$. First $H_{0}$ is positive definite, and from the boundary estimates (Lemma \ref{lem:boundary}) we have $H_{\delta}$ is positive definite in an $\varepsilon$ neighborhood of $\partial K_{\delta}$. From the a priori estimates of the solution $u$ on the Hessian equation \cite{wang94}, we know this bounded depends only on the uniformly bounded geometry of $K_{t}$ which depends on the geometry $K$ and $t$. We conclude that if $v(\cdot,s)$ is strictly convex for all $0\leq s<t$, then $v(\cdot,t)$ is convex.

So if for some $\delta$, $0<\delta<1$, $H_{\delta}$ is positive semi-definite but not positive definite in $K_{\delta}$, we say it is impossible by constant rank theorem (Theorem \ref{thm:revised}) and boundary estimates (Lemma \ref{lem:boundary}). We conclude $H_{\delta}$ is positive definite. Then $v=-\log(-u)$ is strictly convex in $K$.
\end{proof}

\section{Proof of Theorem \ref{thm:bm}}

The aim of this section is to prove Theorem \ref{thm:bm}. Now we state some elementary propositions on the convexity of the matrix functions.

First we recall Jensen's inequality for means (see \cite{brascamp1976extensions}). If $a,b$ are real positive numbers, $\alpha\in[-\infty,+\infty]$ and $\lambda\in(0,1)$, we define
\[
m_{\alpha}(a,b,\lambda)=
\begin{cases}
[(1-\lambda)a^{\alpha}+\lambda b^{\alpha}]^{1/\alpha}, & \alpha\in(-\infty,0)\cup(0,+\infty),\\
\min(a,b), & \alpha=-\infty,\\
a^{1-\lambda}b^{\lambda}, & \alpha=0,\\
\max(a,b), & \alpha=+\infty.
\end{cases}
\]
Jensen's inequality for means implies that
\[
m_{\alpha}(a,b,\lambda)\leq m_{\beta}(a,b,\lambda)\quad\text{if }\alpha\leq\beta.
\]
In particular, the arithmetic--geometric mean inequality holds
\[
a^{1-\lambda}b^{\lambda}\leq (1-\lambda)a+\lambda b,\qquad a,b\geq0,\ \lambda\in[0,1].
\]

\begin{proposition}\label{prop:s2-convex}
$\sigma_{2}(A^{-1})^{1/2}$ is convex in $A\in S^{n}_{++}$.
\end{proposition}

\begin{proof}
This is a special case of Theorem 15.16 in \cite{lieberman1996second}.
\end{proof}

\begin{proof}[Proof of Theorem \ref{thm:bm}]
For $i=0,1$, let $K_{i}$ be a convex domain in $\mathbb{R}^{n}$, and let $u_{i}$ be the solution of
\[
\begin{cases}
\sigma_{2}(D^{2}u_{i})=\lambda(K_{i})(-u_{i})^{2},\quad u_{i}<0 & \text{in }\operatorname{int}(K_{i}),\\
u_{i}=0 & \text{on }\partial K_{i}.
\end{cases}
\]
Then the function $v_{i}(x)=-\log(-u_{i}(x))$ solves
\begin{equation}\label{eq:vi-equation}
\begin{cases}
\sigma_{2}(D^{2}v_{i})-\operatorname{Tr}(P(\nabla v_{i})D^{2}v_{i})=\lambda(K_{i}) & \text{in }\operatorname{int}(K_{i}),\\
v_{i}(x)\to+\infty & x\to\partial K_{i},
\end{cases}
\end{equation}
where $P=(P_{ij})$ with $P_{ij}=|\nabla v|^{2}\delta_{ij}-v_{i}v_{j}$. In Section 4 we have proved $v_{i}$ is strictly convex in $K_{i}$, so that
\[
\det(D^{2}v_{i}(x))>0,\qquad x\in\operatorname{int}(K_{i}).
\]
By the boundary condition verified by $v_{i}$, we have
\[
\nabla v_{i}(\operatorname{int}(K_{i}))=\mathbb{R}^{n}.
\]
Let us now consider the conjugate function $v^{*}_{i}$ of $v_{i}$:
\[
v^{*}_{i}(\rho)=\sup_{x\in K_{i}}\{(x,\rho)-v_{i}(x)\},\qquad \rho\in\mathbb{R}^{n}.
\]
For the basic properties of this function we refer to \cite{rockafellar1970convex}. As $v_{i}$ is strictly convex, $v^{*}_{i}$ is defined on the whole $\mathbb{R}^{n}$, and $\nabla v^{*}_{i}$ is the inverse map of $\nabla v_{i}$:
\[
x=\nabla v^{*}_{i}(\nabla v_{i}(x)),\qquad x\in K_{i}.
\]
In particular,
\begin{equation}\label{eq:legendre-hessian}
D^{2}v_{i}(x)=\bigl[D^{2}v^{*}_{i}(\nabla v_{i}(x))\bigr]^{-1}.
\end{equation}

Let $t\in[0,1]$ and define $K_{t}=(1-t)K_{0}+tK_{1}$. Now introduce a new function $w$ in $K_{t}$ by
\begin{equation}\label{eq:w-def}
w(z)=\min\{(1-t)v_{0}(x)+tv_{1}(y):x\in K_{0},\,y\in K_{1},\,(1-t)x+ty=z\}.
\end{equation}
The function $w$ is called the infimal convolution of $v_{0}$ and $v_{1}$ in $K_{t}$. It is a strictly convex function, and from the boundary conditions in problem \eqref{eq:vi-equation} it can be deduced that
\[
\lim_{z\to\partial K_{t}}w(z)=+\infty.
\]
Moreover
\begin{equation}\label{eq:w-star}
w^{*}=(1-t)v^{*}_{0}+tv^{*}_{1}\quad\text{in }\mathbb{R}^{n}.
\end{equation}

Let $z\in K_{t}$. By the definition of $w$ and the boundary conditions in \eqref{eq:vi-equation}, there exist unique $x\in\operatorname{int}(K_{0})$ and $y\in\operatorname{int}(K_{1})$ such that $z=(1-t)x+ty$ and
\[
w(z)=(1-t)v_{0}(x)+tv_{1}(y).
\]
By the Lagrange multipliers theorem,
\[
\nabla v_{0}(x)=\nabla v_{1}(y)=:\rho.
\]
On the other hand,
\[
\nabla w^{*}(\rho)=(1-t)\nabla v^{*}_{0}(\rho)+t\nabla v^{*}_{1}(\rho)=(1-t)x+ty=z=\nabla w^{*}(\nabla w(z)).
\]
Hence by the injectivity of $\nabla w$ we have $\nabla w(z)=\rho$. Therefore,
\begin{equation}\label{eq:w-hessian}
D^{2}w(z)=\left[(1-t)(D^{2}v_{0}(x))^{-1}+t(D^{2}v_{1}(y))^{-1}\right]^{-1}.
\end{equation}
Also
\[
P(\nabla w)=P(\nabla v_{0})=P(\nabla v_{1})\equiv P(\rho)=:P.
\]

We claim that
\begin{equation}\label{eq:claim-a}
\sigma_{2}(D^{2}w(z))-\operatorname{Tr}(PD^{2}w(z))\leq \max_{i\in\{0,1\}}\lambda(K_{i}),\qquad z\in K_{t}.
\end{equation}
Assume the claim. Define $\overline{u}(z):=-e^{-w(z)}$ in $K_{t}$. Then
\[
\begin{cases}
\sigma_{2}(D^{2}\overline{u})\leq \max_{i\in\{0,1\}}\lambda(K_{i})(-\overline{u})^{2},\quad \overline{u}<0 & \text{in }\operatorname{int}(K_{t}),\\
\overline{u}=0 & \text{on }\partial K_{t}.
\end{cases}
\]
Multiplying both sides of the inequality by $-\overline{u}$ and integrating over $K_{t}$ gives
\[
\max_{i\in\{0,1\}}\lambda(K_{i})
\geq
\frac{-\int_{K_{t}}\overline{u}S_{2}(D^{2}\overline{u})\,dx}
{\int_{K_{t}}|\overline{u}|^{3}\,dx}
\geq \lambda(K_{t}).
\]
Hence
\begin{equation}\label{eq:max-lambda}
\max\{\lambda(K_{0}),\lambda(K_{1})\}\geq \lambda((1-t)K_{0}+tK_{1}),\qquad t\in[0,1].
\end{equation}
In order to get the Brunn--Minkowski inequality, replace $K_{0}$ with $K'_{0}$, $K_{1}$ with $K'_{1}$ and $t$ with $t'$ where
\[
K'_{0}=[\lambda(K_{0})]^{1/4}K_{0},\qquad
K'_{1}=[\lambda(K_{1})]^{1/4}K_{1},
\]
and
\[
t'=\frac{t[\lambda(K_{1})]^{-1/4}}
{(1-t)[\lambda(K_{0})]^{-1/4}+t[\lambda(K_{1})]^{-1/4}}.
\]
This yields \eqref{eq:bm-s2}.

It remains to prove \eqref{eq:claim-a}. Since $w$ is strictly convex in $K_{t}$, there exists a unique point $z_{0}\in K_{t}$ such that $Dw(z_{0})=0$.

At the unique point $z_{0}\in K_{t}$ with $Dw(z_{0})=0$, there are unique $x_{0}\in K_{0}$ and $y_{0}\in K_{1}$ such that $z_{0}=(1-t)x_{0}+ty_{0}$ and $Dv_{0}(x_{0})=Dv_{1}(y_{0})=0$. Moreover,
\[
\lambda(K_{0})=\sigma_{2}(D^{2}v_{0})(x_{0}),\qquad
\lambda(K_{1})=\sigma_{2}(D^{2}v_{1})(y_{0}).
\]
Using \eqref{eq:w-hessian} and Proposition \ref{prop:s2-convex}, we have
\[
\sigma_{2}(D^{2}w(z_{0}))^{1/2}
\leq (1-t)\sigma_{2}(D^{2}v_{0}(x_{0}))^{1/2}
+t\sigma_{2}(D^{2}v_{1}(y_{0}))^{1/2},
\]
which implies \eqref{eq:claim-a} at $z_{0}$.

At a point $z\in K_{t}$ with $Dw(z)\neq0$, there are unique $x\in K_{0}$ and $y\in K_{1}$ such that $z=(1-t)x+ty$ and $Dv_{0}(x)=Dv_{1}(y)\neq0$. Using Lemma\ref{thm:main} and Lemma\ref{lem:inverse-convex}, we have
\begin{equation}\label{eq:prop-use}
\frac{\sigma_{2}(D^{2}w(z))}{\operatorname{Tr}(PD^{2}w(z))}
\leq
(1-t)\frac{\sigma_{2}(D^{2}v_{0}(x))}{\operatorname{Tr}(PD^{2}v_{0}(x))}
+t\frac{\sigma_{2}(D^{2}v_{1}(y))}{\operatorname{Tr}(PD^{2}v_{1}(y))}.
\end{equation}
Consequently it follows from \eqref{eq:vi-equation} that
\begin{align*}
\frac{\sigma_{2}(D^{2}w(z))}{\operatorname{Tr}(PD^{2}w(z))}
&\leq
(1-t)\frac{\lambda(K_{0})}{\operatorname{Tr}(PD^{2}v_{0}(x))}
+t\frac{\lambda(K_{1})}{\operatorname{Tr}(PD^{2}v_{1}(y))}+1\\
&\leq
\max_{i\in\{0,1\}}\lambda(K_{i})
\left[
\frac{1-t}{\operatorname{Tr}(PD^{2}v_{0}(x))}
+\frac{t}{\operatorname{Tr}(PD^{2}v_{1}(y))}
\right]+1\\
&\leq
\max_{i\in\{0,1\}}\lambda(K_{i})
\frac{1}{\operatorname{Tr}(PD^{2}w(z))}+1,
\end{align*}
where the last inequality still comes from Proposition \ref{prop:matrix-convex}. Thus \eqref{eq:claim-a} follows.

Up to now, we complete the proof of the Brunn--Minkowski inequality.

We now deal with the equality case. If $K_{0}$ is homothetic to $K_{1}$, then the equality holds in \eqref{eq:bm-s2} by the homogeneity of $\lambda$ and by the invariance with respect to translation.

Conversely, if equality holds in \eqref{eq:bm-s2}, then the arguments above show that equality must hold in \eqref{eq:max-lambda}, up to a normalization of the involved sets. Namely, let $K'_{0}$, $K'_{1}$ and $t'$ be as above and let $K'_{t}=(1-t')K'_{0}+t'K'_{1}$. Thanks to the homogeneity of $\lambda$ we may assume
\[
\lambda(K'_{t})=\lambda(K'_{0})=\lambda(K'_{1})=1.
\]
Hence reduce the equality to the case in which the bodies $K_{0},K_{1}$ and $K_{t}$ have the same eigenvalue 1.

We shall prove that for $x\in K_{0}$ and $y\in K_{1}$ such that $z=(1-t)x+ty$ and $Dv_{0}(x)=Dv_{1}(y)=Dw(z)$, we have
\begin{equation}\label{eq:claim-b}
D^{2}v_{0}(x)=D^{2}v_{1}(y).
\end{equation}
Assuming \eqref{eq:claim-b}, as in Colesanti \cite{colesanti2005brunn}, we conclude that
\[
D^{2}v^{*}_{0}(\rho)=D^{2}v^{*}_{1}(\rho)\quad \forall \rho\in\mathbb{R}^{n},
\]
and hence
\[
\nabla v^{*}_{0}(\rho)=\nabla v^{*}_{1}(\rho)+\overline{\rho}\quad \forall \rho\in\mathbb{R}^{n}
\]
for some fixed $\overline{\rho}\in\mathbb{R}^{n}$. Finally
\[
K_{0}=\nabla v^{*}_{0}(\mathbb{R}^{n})=\nabla v^{*}_{1}(\mathbb{R}^{n})+\overline{\rho}=K_{1}+\overline{\rho}.
\]

It remains to prove \eqref{eq:claim-b}. At the unique point $z_{0}\in K_{t}$ with $Dw(z_{0})=0$, there are unique $x_{0}\in K_{0}$ and $y_{0}\in K_{1}$ such that $z_{0}=(1-t)x_{0}+ty_{0}$ and $Dv_{0}(x_{0})=Dv_{1}(y_{0})=0$. From \eqref{eq:w-hessian} and the equality case in Proposition \ref{prop:s2-convex}, we have
\[
\sigma_{2}(D^{2}w(z_{0}))^{1/2}
=(1-t)\sigma_{2}(D^{2}v_{0}(x_{0}))^{1/2}
+t\sigma_{2}(D^{2}v_{1}(y_{0}))^{1/2}
\]
if and only if $D^{2}v_{0}(x_{0})=D^{2}v_{1}(y_{0})$.

At a point $z\in K_{t}$ with $Dw(z)\neq0$, there are $x\in K_{0}$ and $y\in K_{1}$ such that $z=(1-t)x+ty$ and $Dv_{0}(x)=Dv_{1}(y)\neq0$. Then all the inequalities in \eqref{eq:prop-use} and the following estimates become equalities. Let
\[
A=(D^{2}v_{0}(x))^{-1},\qquad B=(D^{2}v_{1}(y))^{-1}.
\]
The equality case in Section\ref{sec equal} and Remark \ref{rem:inverse-equality} gives
\[
A=cB
\]
Now the equality in \eqref{eq:prop-use} immediately implies $A=B$, i.e.
\[
D^{2}v_{0}(x)=D^{2}v_{1}(y),
\]
thanks to the homogeneity of degree $-1$ in $\sigma_{2}(A^{-1})/\operatorname{Tr}(PA^{-1})$.

This proves \eqref{eq:claim-b} and finishes the proof of Theorem \ref{thm:bm}.
\end{proof}
\small
\bibliographystyle{alpha}
\bibliography{reference}

{\small
\indent 
(Xi-Nan Ma) Department of Mathematics, University of Science and Technology of China, Hefei, 230026, Anhui Province, China.\;
Email address: xinan@ustc.edu.cn\\ \\
(Jiahuan Li) Department of Mathematics, University of Science and Technology of China, Hefei, 230026, Anhui Province, China.\;
Email address: jiahuan@mail.ustc.edu.cn\\ \\
(Paolo Salani) Dip.to di Matematica e Informatica "U. Dini", Universit\`a degli Studi di Firenze, 50134 Firenze, Italy.
Email address: paolo.salani@unifi.it
}

\end{document}